\documentclass{amsart}[10pt]
\usepackage[matrix,arrow,curve]{xy}
\usepackage{amssymb, amsmath, euscript}
\theoremstyle{definition}\newtheorem{theorem}{Theorem}
\newtheorem*{theorem*}{theorem}
\newtheorem{proposition}{Proposition}
\newtheorem{lemma}{Lemma}
\newtheorem{corollary}{Corollary}
\theoremstyle{definition}\newtheorem{definition}{Definition}
\theoremstyle{remark}\newtheorem{example}{Example}

\renewcommand{\to}[1][]{\xrightarrow{#1}}
\newcounter{AP}
\begin{document}
\title[Bounded reductive subalgebras]{Bounded reductive subalgebras of \Large$\frak{sl}_n$}
\author{Alexey Vladimirovich Petukhov}
\address{Chair of High Algebra of Mechanical and Mathematical Department,  Moscow State University named after ~M.~V.~Lomonosov, Moscow 119992 and 
Jacobs University Bremen, School of Science and Engineering, Bremen D-28759}
\email{a.petukhov@jacobs-university.de}
\maketitle
\begin{abstract}Let $\frak g$ be a reductive Lie algebra and $\frak k\subset\frak g$ be a reductive in $\frak g$ subalgebra. A $(\frak g, \frak k)$-module $M$ is a $\frak g$-module for which any element $m\in M$ is contained in a finite-dimensional $\frak k$-submodule of $M$. We say that a $(\frak g, \frak k)$-module $M$ is bounded if there exists a constant $C_M$ such that the Jordan-H\"older multiplicities of any simple finite-dimensional $\frak k$-module in every finite-dimensional $\frak k$-submodule of $M$ are bounded by $C_M$. In the present paper we describe explicitly all reductive in $\frak{sl}_n$ subalgebras $\frak k$ which admit a bounded simple infinite-dimensional $(\frak{sl}_n, \frak k)$-module. Our technique is based on symplectic geometry and the notion of spherical variety. We also characterize the irreducible components of the associated varieties of simple bounded $(\frak g, \frak k)$-modules.\end{abstract}
\section{Introduction}
 Throughout this paper $\frak g$ will be a reductive Lie algebra and $\frak k\subset\frak g$ will be a reductive in $\frak g$ subalgebra. Recall that a $(\frak g, \frak k)$-{\it module} is a $\frak g$-module which is {\it locally finite as a $\frak k$-module}, i.e. the space generated by $m, km, k^2m,...$ is finite-dimensional for all $m\in M$ and $k\in\frak k$. If a $(\frak g, \frak k)$-module is simple, it is a direct sum of simple finite-dimensional $\frak k$-modules. There are two well-known categories of $(\frak g, \frak k)$-modules: the category of Harish-Chandra modules and the category O. In the first case $\frak k$ is a symmetric subalgebra of $\frak g$ (i.e. $\frak k$ coincides with the fixed points of an involution of $\frak g$), and in the second case $\frak k$ is a Cartan subalgebra $\frak h_\frak g$ of $\frak g$. In both cases the $(\frak g, \frak k)$-modules in question have the important additional property that they have finite $\frak k$-multiplicities, i.e. have finite-dimensional $\frak k$-isotypic components.

I.~Penkov, V.~Serganova, and G.~Zuckerman have proposed to study, and attempt to classify, simple $(\frak g, \frak k)$-modules with finite $\frak k$-multiplicities for arbitrary reductive in $\frak g$ subalgebras $\frak k$,~\cite{PSZ},~\cite{PZ}. Such classifications are known for Harish-Chandra modules, see~\cite{KV} and references therein. In the classification of simple $(\frak g, \frak h_\frak g)$-modules  of finite type the bounded simple modules play a crucial role. Based on this, and on the experience with Harish-Chandra modules, I.~Penkov and V.~Serganova have proposed to study bounded $(\frak g, \frak k)$-modules for general reductive subalgebras $\frak k$, i.e. $(\frak g, \frak k)$-modules whose multiplicities are uniformly bounded. A question arising in this context is, given $\frak g$, to describe all reductive in $\frak g$ bounded subalgebras, i.e. reductive in $\frak g$ subalgebras $\frak k$ for which at least one infinite-dimensional bounded $(\frak g, \frak k)$-module exists. In~\cite{PS} I.~Penkov and V.~Serganova gave a partial answer to this problem, and in particular proved an important inequality which restricts severely the class of possible $\frak k$. They also gave the complete list of bounded reductive subalgebras of $\frak g=\frak{sl}_n$ which are maximal subalgebras.

In the present paper we describe explicitly all reductive in $\frak{sl}_n$ bounded subalgebras. Our technique is based on symplectic geometry and the notion of spherical variety. We also characterize the irreducible components of the support varieties of simple bounded $(\frak g, \frak k)$-modules.
\section{Basic definitions and statement of results}
 The ground field $\mathbb F$ is algebraically closed and of characteristic 0. We work in the category of algebraic varieties over $\mathbb F$. All Lie algebras considered are finite-dimensional. In what follows we use the term $\frak k$-{\it type} for a simple finite-dimensional $\frak k$-module. By T$X$ we denote the total space of the tangent bundle $\EuScript TX$ of a smooth variety $X$, and by T$_xX$ the tangent space to $X$ at a point $x\in X$.
\begin{definition}\upshape For a $\frak k$-module $M$ and a $\frak k$-type $W_\lambda$, we define the $\frak k$-{\it{}multiplicity of $W_\lambda$ in }$M$ as the supremum of the  Jordan-H\"older multiplicities of $W_\lambda$ in all finite-dimensional $\frak k$-submodules of $M$.\end{definition}
\begin{definition}\upshape A $(\frak g, \frak k)$-module $M$ is called {\it bounded} if $M$ is {\it a bounded $\frak k$-module}, i.e.  $\frak k$-multiplicities of all $\frak k$-types in $M$ are bounded by the same constant $C_M>0$. A bounded $(\frak g, \frak k)$-module $M$ is {\it{}multiplicity-free} if $C_M$ can be chosen as 1.\end{definition}
\begin{definition}\upshape We say that a variety $X$ is a $\frak k$-{\it variety} if a homomorphism of Lie algebras $\tau_X:\frak k\xrightarrow{} \EuScript TX$ is given. We say that $\frak k$ {\it has an open orbit on }$X$ if there exists a point $x\in X$ such that the homomorphism $\tau_X|_x:\frak k\to\mathrm T_xX$ is surjective.\end{definition}
\begin{definition}\upshape Let $X$ be a $\frak k$-variety. Then $X$ is $\frak k$-{\it{}spherical} if and only if a Borel subalgebra of $\frak k$ has an open orbit on any irreducible component of $X$.\end{definition}
We should mention that this definition of a spherical variety is slightly different from the one given in~\cite{VK}.

The following theorem is the main result of this paper.
\begin{theorem}\label{T1} Assume that $\frak g=\frak{sl}_n$. A pair $(\frak{sl}_n, \frak k)$ admits an infinite-dimensional simple bounded $(\frak{sl}_n, \frak k)$-module if and only if the projective space $\mathbb P(\mathbb F^n)$ is a spherical $\frak k$-variety.\end{theorem}
Any finitely generated $\frak g$-module $M$ has an associated graded S$(\frak g)$-module $\overline{\mathrm{gr}}M$. We denote by V$(M)$ the support of $\overline{\mathrm{gr}}M$ (for the precise definitions see the subsequent two sections). As a step towards the proof of Theorem~\ref{T1} we establish the following result which might be of interest on it's own.
\begin{proposition}\label{Pav}a) A finitely generated $(\frak g, \frak k)$-module $M$ is bounded if and only if its support variety $\mathrm V(M)$ is $\frak k$-spherical.\\ b) If the equivalent conditions of a) are satisfied any irreducible component $\tilde V$ of $\mathrm V(M)$ is a conical Lagrangian subvariety of \begin{center}$G\tilde V:=\{x\in\frak g^*\mid x=gv$ for some $g\in G$ and $v\in\tilde V$\}.\end{center}\end{proposition}
All finite-dimensional $\frak k$-modules $W$ such that $\mathbb P(W)$ is a $\frak k$-spherical variety are known from the works of V.~Kac~\cite{KC}, C.~Benson and G.~Ratcliff~\cite{BR}, A. Leahy~\cite{Le}. The list of respective pairs $(\frak k, W)$ is reproduced in the Appendix. Theorem~\ref{T1} implies the following.
\begin{corollary}The list of pairs $(\frak{sl}(W), \frak k)$ for which $\frak k$ is reductive and bounded in $\frak{sl}(W)$ coincides with the list of Benson-Ratcliff and Leahy reproduced in the Appendix.\end{corollary}
We also prove the Conjecture 6.6 of~\cite{PS}: there exists an infinite-dimensional bounded $(\frak{sl}_n, \frak k)$-module if and only if there exists a multiplicity-free simple infinite-dimensional $(\frak{sl}_n, \frak k)$-module.

\section{Preliminaries on symplectic geometry}
In what follows we denote by $G$ the adjoint group of $[\frak g, \frak g]$, and by $K$ the connected subgroup of $G$ with Lie algebra $\frak k\cap[\frak g, \frak g]$. By T$^*X$ we denote the total space of the cotangent bundle of $X$ and by T$^*_xX$ --- the space dual to T$_xX$.
\begin{definition}\upshape Suppose that $X$ is a smooth variety which admits a closed nondegenerate 2-form $\omega$. Such a pair $(X,\omega)$ is called a {\it symplectic variety}. If $X$ is a $G$-variety and $\omega$ is $G$-invariant, $(X, \omega)$ is called a {\it symplectic $G$-variety}.\end{definition}
\begin{example}\upshape Let $X$ be a smooth $G$-variety. Then $\mathrm T^*X$ has a one-form $\alpha_X$ defined at a point $(l,x) (l\in\mathrm T_x^*X)$ by the equality $\alpha_X(\xi)=l(\pi_*\xi)$ for any $\xi\in\mathrm T_{(l, x)}(\mathrm T^*X)$, where $\pi: \mathrm T^*X\to X$ is the projection. The differential $\mathrm d\alpha_X$ is a nondegenerate $G$-invariant two-form on $\mathrm T^*X$ and therefore $(\mathrm T^*X, \mathrm d\alpha_X)$ is a symplectic $G$-variety.\end{example}
\begin{example}\upshape Let $\EuScript O$ be a $G$-orbit in $\frak g^*$. Then $\EuScript O$ has a {\it{}Kostant-Kirillov 2-form } $\omega(\cdot~\!, \cdot)$ defined at a point $x\in\frak g^*$ by the equality $\omega_x(\tau_{\frak g^*}p|_x, \tau_{\frak g^*}q|_x)=x([p,q])$ for $p, q\in\frak g$.\end{example}
\begin{definition}\upshape Let $(X,\omega)$ be a symplectic variety. We call a subvariety $Y\subset X$\\a) {\it isotropic} if $\omega|_{\mathrm T_yY}=0$ for a generic point $y\in Y$;\\b) {\it coisotropic} if $\omega|_{(\mathrm T_yY)^{\bot_\omega}}=0$ for a generic point $y\in Y$;\\c) {\it Lagrangian} if T$_yY=($T$_yY)^{\bot_\omega}$ for a generic point $y\in Y$ or equivalently if it is both isotropic and coisotropic.\end{definition}
\begin{example}\upshape \label{TZX} Let $X$ be a smooth variety and $Y\subset X$ be a smooth subvariety. Then the total space N$^*_{Y/X}$ of the conormal bundle to $Y$ in $X$ is Lagrangian in T$^*X$.\end{example}
\begin{proposition}[{see for example~\cite[Lemma 1.3.27]{NG}}]\label{NG} Any closed irreducible conical (i.e. $\mathbb F^*$-stable) Lagrangian $G$-subvariety of $\mathrm T^*X$ is the closure of the total space $\mathrm N^*_{Y/X}$ of the conormal bundle to a $G$-subvariety $Y\subset X$.\end{proposition}
Let $X$ be a $G$-variety. The map\begin{center} T$^*X\times\frak g\to\mathbb F$\hspace{40pt}$((x, l), g)\to l(\tau_X g|_x),$\end{center} where $g\in \frak g, x\in X, l\in$T$^*_xX$, induces a map $\phi: $T$^*X\to\frak g^*$ called the {\it{}moment map}. This map provides the following description of an orbit $ Gu\subset\frak g^*$ such that $0\in\overline{Gu}$ (in what follows we call such orbits {\it nilpotent)}. Suppose $P$ is a parabolic subgroup of $G$.
\begin{theorem}[R.~Richardson~\cite{Ri}]\label{MMapSl}  The moment map $\phi_P:\mathrm T^*(G/P)\to\frak g^*$ is a proper morphism to the closure of some nilpotent orbit $\overline{Gu}$ and is a finite morphism over $Gu$.\end{theorem}
\begin{example}\upshape\label{Pn}Let $G=\mathrm{SL}_n$ and $G/P=\mathbb P(\mathbb F^n)$. Then $\phi_P(\mathrm T^*(G/P))$ (considered as a subset of $\frak{sl}_n$) coincides with the set of nilpotent matrices of rank $\le$1. The SL$_n$-orbit open in $\phi_P(\mathrm T^*(G/P))$ is the set of $\mathrm{SL}_n$-highest weight vectors in $\frak{sl}_n^*$ and is contained in the closure of any nonzero nilpotent $\mathrm{SL}_n$-orbit in $\frak{sl}_n^*$.\end{example}
For $G=$SL$_n$, each moment map $\phi_P$ corresponding to a parabolic subgroup $P$ is a birational isomorphism of $\mathrm T^*(G/P)$ with the image of $\phi_P$, and one can obtain the closure of any nilpotent orbit $\overline{Gu}$ as the image of a suitable moment map $\phi_P$. 
Moreover the canonical symplectic form of T$^*(G/P)$ coincides on an open set with the pullback of the Kostant-Kirillov form of $\phi_P($T$^*(G/P))$.
\section{Preliminaries on $\frak g$-modules}
Let $\{U_i\}_{i\in\mathbb Z_{\ge 0}}$ be the standard filtration on $U:=$U$(\frak g)$. Suppose that $M$ is a U$(\frak g)$-module and that a filtration $\cup_{i\in \mathbb Z_{\ge 0}}M_i$ of vector spaces is given. We say that this filtration is {\it good} if\begin{center}(1) $U_i M_j= M_{i+j}$;\hspace{40pt}(2) dim~$M_i<\infty~$for all $i\in\mathbb Z_{\ge 0}$.\end{center}
Such a filtration arises from any finite-dimensional space of generators $M_0$. The corresponding associated graded object gr$M=\oplus_{i\in\mathbb Z_{\ge0}}M_{i+1}/M_i$ is a module over gr~U($\frak{g})\cong$S$(\frak g)$, and we set
\begin{center}J$_M:=\{s\in $S$(\frak g)$ $\mid$ there exists $k\in \mathbb Z_{\ge0}$ such that $s^km=0$ for all $m\in $gr$M\}.$\end{center}
In this way we associate to any $\frak g$-module $M$ the variety
\begin{center}V($M):=\{x\in\frak g^*\mid f(x)=0$ for all $f\in$J$_M\}$.\end{center}
It is easy to check that the module gr$M$ depends on the choice of good filtration, but the ideal J$_M$  and the variety V($M)$ does not (see for instance~\cite{Gab}). Let N$_G(\frak g^*)\subset\frak g^*$ be the union of all nilpotent orbits.
\begin{lemma}[\cite{Dix}] Let $M$ be a simple $\frak g$-module. Then $\mathrm V(M)\subset\mathrm N_G(\frak g^*)$.\end{lemma}
\begin{theorem}[O.~Gabber \cite{Gab}]Any irreducible component $\tilde V$ of $\mathrm V(M)$ is coisotropic inside a unique open $G$-orbit of G$\tilde V:=\{x\in\frak g^*\mid x=gv$ for some $g\in G, v\in\tilde V\}$.\end{theorem}
\begin{corollary}[S.~Fernando \cite{F}] Set $\mathrm V(M)^\bot:=\{g\in\frak g\mid v(g)=0$ for all $v\in \mathrm{V(}M\mathrm )\subset\frak g^*\}$. Then $\mathrm V(M)^\bot$ is a Lie algebra and $\mathrm V(M)$ is a $\mathrm V(M)^\bot$-variety.\end{corollary}
\begin{theorem}[S.~Fernando~{\cite[Cor. 2.7]{F}}, V.~Kac~\cite{KC2}]Set $\frak g[M]:=\{g\in\frak g\mid\mathrm{dim}(\mathrm{span}_{i\in\mathbb Z_{\ge 0}}\{ g^im\})< \infty$ for all $m\in M\}$. Then $\frak g[M]$ is a Lie algebra and $\frak g[M]\subset\mathrm V(M)^\bot$.\end{theorem}
\begin{corollary}Let $M$ be a $(\frak g, \frak k)$-module. Then $\mathrm{V(}M\mathrm)\subset\frak k^\bot$ and $\mathrm{V(}M\mathrm)$ is a $\frak k$-variety.\end{corollary}
\section{Proof of Theorem~\ref{T1}}
Let $M$ be a $(\frak g, \frak k)$-module and $M_0$ be a $\frak k$-stable finite-dimensional space of generators of $M$; $\mathrm J_M,$ gr$M$ be the corresponding objects costructed as in Section 4. Consider the S$(\frak g)$-modules\begin{center}  J$_M^{-i}\{0\}:=\{m\in$gr$M\mid j_1...j_im=0$ for all $j_1,...,j_i\in$J$_M~\}$.\end{center} One can easily see that these modules form an ascending filtration of gr$M$ such that \begin{center}$\cup_{i=1}^\infty$J$_M^{-i}\{0\}=$gr$M$.\end{center} Since S$(\frak g)$ is a N\"otherian ring, the filtration stabilizes, i.e. J$_M^{-i}\{0\}=$gr$M$ for some $i$. By $\overline{\mathrm{gr}}M$ we denote the corresponding graded object. By definition, $\overline{\mathrm{gr}}M$ is an S$(\frak g)/$J$_M$-module. Suppose that $f\overline{\mathrm{gr}}M=0$ for some $f\in$S$(\frak g)$. Then $f^i$gr$M=0$ and hence $f\in$J$_M$. This proves that the annihilator of $\mathrm{\overline{gr}}M$ in $\mathrm S(\frak g)/\mathrm J_M$ equals zero.

The following lemma is a reformulation in the terms of the present paper of a result of \'E.~Vinberg and B.~Kimelfeld~\cite[Thm.~2]{VK}.
\begin{lemma}A quasiaffine algebraic $K$-variety $X$ is $\frak k$-spherical if and only if the space of regular functions $\mathbb F[X]$ is a bounded $\frak k$-module. If the variety $X$ is irreducible, then $\mathbb F[X]$ is a multiplicity-free $\frak k$-module.\end{lemma}
\begin{theorem}[D.~Panyushev {~\cite[Thm 2.1]{Pan}}]\label{Pan}Let $X$ be a smooth $G$-variety and $M$ a smooth locally closed $G$-stable subvariety. Then, for any Borel subgroup $B\subset G$, the generic stabilizers of the actions of $B$ on $X,\mathrm N_{M/X}$ and $\mathrm N_{M/X}^*$ are isomorphic.\label{MPan}\end{theorem}
\hspace{-11pt}{\bf Proposition~\ref{Pav}.} {\it a) The module $M$ is bounded if and only if its support variety $\mathrm V(M)$ is $\frak k$-spherical.\\ b) If the equivalent conditions of a) are satisfied any irreducible component $\tilde V$ of $\mathrm V(M)$ is a conical Lagrangian subvariety of $G\tilde V$.}
\begin{proof}a)As $M$ is a finitely generated $\frak g$-module, the S$(\frak g)$-modules gr$M$ and $\mathrm{\overline{gr}}M$ are finitely generated. Let $\tilde M_0$ be a $\frak k$-stable finite-dimensional space of generators of $\mathrm{\overline{gr}}M$. Then there is a surjective homomorphism \begin{center}$\psi: \tilde M_0\otimes_\mathbb F$(S$(\frak g)/$J$_M)\to\overline{\mathrm{gr}}M$.\end{center} Suppose that the variety V$(M)$ is $\frak k$-spherical. Then $\tilde M_0\otimes_\mathbb F$(S$(\frak g)/$J$_M$) is a bounded $\frak k$-module. Therefore $\overline{\mathrm{gr}}M$ is bounded, which implies that $M$ is a bounded $\frak k$-module too.

Assume now that a $\frak g$-module $M$ is $\frak k$-bounded. Set\begin{center}Rad$M=\{m\in\overline{\mathrm{gr}}M\mid fm=0$ and $f\ne 0$ for some $f\in$S$(\frak g)/$J$_M\}$.\end{center} Then Rad$M$ is a proper $\frak k$-stable submodule of $\overline{\mathrm{gr}}M$ and $\tilde M_0\not\subset$Rad$M$. The homomorphism $\psi$ induces the injective homomorphism $\hat\psi:$S$(\frak g)/$J$_M\to\tilde M_0^*\otimes_\mathbb F\overline{\mathrm{gr}}M$. Therefore S$(\frak g)/$J$_M$ is a $\frak k$-bounded module and V($M$) is a $\frak k$-spherical variety. This completes the proof of a).\\
b) Let $\tilde V\subset\mathrm V(M)$ be an irreducible component and $x\in\tilde V$ be a generic point. As $x\in\frak k^\bot$ we have \begin{center}$x([k_1, k_2])=\omega_x(\tau_{\frak g^*}k_1|_x, \tau_{\frak g^*}k_2|_x)=0$\end{center} for all $k_1, k_2\in\frak k$. Therefore any $K$-orbit in $\frak k^\bot$ is isotropic. As $\tilde V$ is a $\frak k$-spherical variety, $\tilde V$ has an open $K$-orbit. Therefore $\tilde V$ is Lagrangian in $G\tilde V$ and  this completes the proof of b).\end{proof}
\begin{theorem}\label{M_to_Gr}Assume that $\frak g=\frak{sl}_n$. If there exists a simple infinite-dimensional bounded $(\frak{sl}_n, \frak k)$-module $M$, then $\mathrm{Gr}(r,\mathbb F^n)$ is a spherical $\frak k$-variety for some $r$.\end{theorem}
\begin{proof} Let $\tilde V$ be an irreducible component of V$(M)$. By the discussion succeeding after Example~\ref{Pn} the variety $G\tilde V ($here $G\cong$SL$_n)$ is birationally isomorphic to T$^*$(SL$_n/P)$ for some parabolic subgroup $P\subset$SL$_n$, and the subvariety $\tilde V\subset G\tilde V$ is isomorphic to a conical Lagrangian subvariety $Y$ of T$^*(G/P)$. Any conical Lagrangian subvariety of T$^*(G/P)$ is the closure of the total space of the conormal bundle N$^*_{Z/(G/P)}$ to some smooth subvariety $Z\subset G/P$, see Proposition~\ref{NG} above. Therefore $\tilde V$ is birationally isomorphic to the total space of the conormal bundle to a smooth subvariety $Z\subset G/P$.

Obviously $Z$ is $\frak k$-stable. Then, by Theorem~\ref{MPan}, the variety $G/P$ has an open orbit of a Borel subalgebra of $\frak k$, i.e. is $\frak k$-spherical. This shows that $G/\tilde P$ is $\frak k$-spherical for any maximal parabolic subgroup $\tilde P$ with $\tilde P\supset P$. Since $G/\tilde P$ is a Grassmannian, the proof is complete.\end{proof}
\begin{theorem}\label{Gr_to_M}Assume that $G=\mathrm{SL}(V)$. Let $P\subset G$ be a parabolic subgroup  such that $G/P$ is $\frak k$-spherical. Then there exists a simple infinite-dimensional multiplicity-free $(\frak g, \frak k)$-module  $M$.\end{theorem}
\begin{proof}Theorem 6.3 in ~\cite{PS} proves the existence of a simple infinite-dimensional multiplicity-free $(\frak g, \frak k)$-module under the assumption that there exists a parabolic subgroup $P\subset G$ for which $K$ has a proper closed orbit on $G/P$ such that the total space of its conormal bundle is $K$-spherical. By~Theorem~\ref{Pan} this latter condition is equivalent to the $K$-sphericity of $G/P$. It remains to consider the case when $G/P$ has no proper closed $K$-orbits on $G/P$, i.e. $K$ has only one orbit on $G/P$. This happens only if $\frak g\cong\frak{sl}_{2n}$, $\frak k\cong\frak{sp}_{2n}$~\cite{On}. However, in this last case the existence of a simple infinite-dimensional multiplicity-free $(\frak g, \frak k)$-module is well known, see for instance~\cite{PS}.\end{proof}
We have thus proved the following weaker version of Theorem~\ref{T1}.
\begin{corollary}\label{Gr1}Assume that $\frak g=\frak{sl}_n$. A pair $(\frak{sl}_n, \frak k)$ admits an infinite-dimensional bounded simple $(\frak{sl}_n, \frak k)$-module if and only if $\mathrm{Gr}(r, \mathbb F^n)$ is a spherical $\frak k$-variety for some $r$.\end{corollary}
\begin{proof} The statement  follows directly from Theorems~\ref{Gr_to_M} and ~\ref{M_to_Gr}.\end{proof}
\begin{corollary}[see also ~\cite{PS}, Conjecture 6.6]Assume that $\frak g=\frak{sl}_n$. If there exists a bounded simple infinite-dimensional $(\frak{sl}_n, \frak k)$-module, then there exists a multiplicity-free simple infinite-dimensional $(\frak{sl}_n, \frak k)$-module.\end{corollary}
\begin{proof}The statement follows directly from Corollary~\ref{Gr1}.\end{proof}
In order to prove Theorem~\ref{T1} it remains to show that $r$ in Corollary~\ref{Gr1} can be chosen to equal 1. For this we need the following result.
\begin{theorem}[I.~Losev ~\cite{Lo}]\label{LI}Suppose $X$ is a strongly equidefectinal~\cite[Def. 1.2.5]{Lo} normal affine irreducible Hamiltonian $G$-variety. Then $\mathbb F(X)^G=\mathrm{Quot}(\mathbb F[X]^G)$, where $\mathrm{Quot}(\mathbb F[X]^G)$ is the field of fractions of $\mathbb F[X]^G$.\end{theorem}

We say that a $K$-symplectic variety $(X, \omega)$ is {\it $K$-coisotropic} if the generic $K$-orbit on $X$ is $K$-coisotropic in $X$. We are now ready to prove the following.
\begin{theorem}Let $V$ be a $K$-module. Suppose $\mathrm{Gr}(r, V)$ is a $K$-spherical variety for some $r<\mathrm{dim}V$. Then $\mathbb P(V)$ is a $K$-spherical variety.\end{theorem}
\begin{proof} The variety Gr$(r, V)$ is $K$-spherical if and only if the variety T$^*$Gr$(r, V)$ is $K$-coisotropic~\cite[Ch.~II, Cor.~1]{Vi}. As T$^*$Gr$(r, V)$ is $K$-birationally isomorphic to some nilpotent orbit $\EuScript O\subset\frak{sl}(V)^*$, this orbit $\EuScript O\subset\frak{sl}(V)^*$ is $K$-coisotropic. Therefore $\mathbb F(\overline{\EuScript O})^K$ is a Poisson-commutative subfield of $\mathbb F(\overline{\EuScript O})$~\cite[Ch.~II, Prop.~5]{Vi}. Since the closure of any SL$(V)$-orbit in $\frak{sl}(V)^*$ is a strongly equidefectinal normal affine irreducible Hamiltonian $K$-variety~\cite[Cor. 3.4.1]{Lo}, $\mathbb F[\overline{\EuScript O}]^K$ is Poisson-commutative.

Let $\EuScript O_{min}$ be the nonzero nilpotent orbit in $\frak{sl}(V)$ of minimal dimension. It is well known that $\EuScript O_{min}\subset\overline{\EuScript O}$. Hence $\mathbb F[\overline{\EuScript O_{min}}]^K$ is a quotient of $\mathbb F[\overline{\EuScript O}]^K$ and $\mathbb F[\overline{\EuScript O_{min}}]^K$ is Poisson-commutative. As $\EuScript O_{min}$ is $K$-birationally isomorphic to T$^*\mathbb P(V)$ (see  Example~\ref{Pn} above), the variety $\mathbb P(V)$ is $K$-spherical.\end{proof}
\section{Appendix. Results of C.~Benson, G.~Ratcliff, A.~Leahy, V.~Kac}
The classification of the spherical modules has been worked out in several steps. V.~Kac has classified the simple spherical modules in~\cite{KC}, C.~Benson and G.~Ratcliff have classified all spherical modules in~\cite{BR}. The classification is contained also in paper~\cite{Le} of A.~Leahy. Below we reproduce their list.

Let $W$ be a $K$-module. Then $W$ is a spherical $\frak k$-variety if and only if the pair $([\frak k, \frak k], W)$ is a direct sum of pairs $(\frak k_i, W_i)$ listed below and in addition $\frak k+\oplus_i c_i=$N$_{\frak{gl}(W)}(\frak k+\oplus_i c_i)$ for certain abelian Lie algebras $c_i$ attached to $(\frak k_i, W_i)$. Here N$_{\frak{gl}(W)}\frak k$ stands for the normalizer of $\frak k$ in $\frak{gl}(W)$ and $c_i$ is a 0-, 1- or 2-dimensional Lie algebra listed in square brackets after the pair $(\frak k_i, W_i)$. This subalgebra is generated by linear operators $h_1$ and $h_{m, n}$. By definition, $h_1=$id. The notation $h_{m, n}$ is used only when $W=W_1\oplus W_2$: in this case $h_{m, n}|_{W_1}=m\cdot$id and $h_{m, n}|_{W_2}=n\cdot$id. The notation "$(\frak k_i, \{W_i, W_i'\})$" is shorthand for "$(\frak k_i, W_i)$" and "$(\frak k_i, W_i')$". Finally, $\omega_i$ stands for the $i$-th fundamental weight and the corresponding fundamental representation. We follow the enumeration convention for fundamental weights of~\cite{OV}.{\tiny
\begin{center}{\bf Table 6.1}: Indecomposable spherical representations.\end{center}
0) $(0, \mathbb F) [0]$.\\
\setcounter{AP}{1}\roman{AP}) Irreducible representations of simple Lie algebras:\\
$\begin{tabular}{lll}$1) (\frak{sl}_n, \{\omega_1, \omega_{n-1}\}) [\mathbb Fh_1] (n\ge 2)$;&2) $(\frak{so}_n, \omega_1) [0] (n\ge 3)$;&3) $(\frak{sp}_{2n}, \omega_1) [\mathbb Fh_1] (n\ge 2)$;\\
4) $(\frak{sl}_n, \{2\omega_1, 2\omega_{n-1}\}) [0] (n\ge 3)$;&5) $(\frak{sl}_{2n+1}, \{\omega_2, \omega_{n-2}\}) [\mathbb Fh_1] (n\ge 2)$;&6) $(\frak{sl}_{2n}, \{\omega_2, \omega_{n-2}\}) [0] (n\ge 3)$;\\7) $(\frak{so}_7, \omega_3) [0]$;&8) $(\frak{so}_8, \{\omega_3, \omega_4\}) [0]$;&9) $(\frak{so}_9, \omega_4) [0]$;\\10) $(\frak{so}_{10}, \{\omega_4, \omega_5\}) [\mathbb Fh_1]$;&11) $($E$_6, \omega_1) [0]$;&12) $($G$_2, \omega_1 ) [0]$.\end{tabular}$\\
\setcounter{AP}{2}\roman{AP}) Irreducible representations of nonsimple Lie algebras:\\
$\begin{tabular}{ll}1) $(\frak {sl}_n\oplus\frak{sl}_m (m>n\ge 2), \{\omega_1, \omega_{n-1}\}\otimes\{\omega_1, \omega_{m-1}\}) [\mathbb Fh_1]$;&2) $(\frak {sl}_n\oplus\frak{sl}_n (n\ge 2), \{\omega_1, \omega_{n-1}\}\otimes\{\omega_1, \omega_{n-1}\}) [0]$;\\3) $(\frak{sl}_2\oplus\frak{sp}_{2n} (n\ge 2), \omega_1\otimes\omega_1) [0]$;&4)
$(\frak{sl}_3\oplus\frak{sp}_{2n} (n\ge 2), \{\omega_1,\omega_2\}\otimes\omega_1) [0]$;\\5) $(\frak{sl}_n\oplus\frak{sp}_4 (n\ge 5), \{\omega_1, \omega_{n-1}\}\otimes\omega_{1}) [\mathbb Fh_1]$;&6) $(\frak{sl}_4\oplus\frak{sp}_4, \{\omega_1, \omega_3\}\otimes\omega_{1}) [0]$.\end{tabular}$\\
\setcounter{AP}{3}\roman{AP}) Reducible representations of Lie algebras:\\
1) $(\frak{sl}_n\oplus\frak{sl}_m\oplus\frak{sl}_2 (n, m\ge 3); (\{\omega_1, \omega_{n-1}\}\oplus\{\omega_1, \omega_{n-1}\})\otimes\omega_1) [\mathbb Fh_{1,0}\oplus\mathbb Fh_{0,1}]$;\\
2) $(\frak{sl}_n (n\ge 3); \{\omega_1\oplus\omega_1, \omega_{n-1}\oplus\omega_{n-1}\}) [\mathbb Fh_{1,1}]$;\\
3) $(\frak{sl}_n (n\ge 3); \omega_1\oplus\omega_{n-1}) [\mathbb Fh_{1,-1}]$;\\
4) $(\frak{sl}_{2n} (n\ge 2); \{\omega_1, \omega_{n-1}\}\oplus\{\omega_2, \omega_{n-2}\}) [\mathbb Fh_{0,1}]$;\\
5) $(\frak{sl}_{2n+1} (n\ge 2); \omega_1\oplus\omega_2) [\mathbb Fh_{1,-m}]$;\\
6) $(\frak{sl}_{2n+1} (n\ge 2); \omega_{n-1}\oplus\omega_2) [\mathbb Fh_{1, m}]$;\\
7) $(\frak{sl}_n\oplus\frak{sl}_m (2\le n<m); \{\omega_1, \omega_{n-1}\}\otimes(\mathbb F\oplus\{\omega_1, \omega_{m-1}\})) [\mathbb Fh_{1,0}]$;\\
8) $(\frak{sl}_n\oplus\frak{sl}_m (m\ge 2, n\ge m+2); \{\omega_1, \omega_{n-1}\}\otimes(\mathbb F\oplus\{\omega_1, \omega_{m-1}\})) [\mathbb Fh_{1, 1}]$;\\
9) $(\frak{sl}_n\oplus\frak{sl}_m (2\le n<m); \{\omega_1, \omega_{n-1}\}\oplus\{\omega_1^* (=\omega_{n-1}), \omega_{n-1}^* (=\omega_1)\}\otimes\{\omega_1, \omega_{m-1}\}) [\mathbb Fh_{1,0}]$;\\
10) $(\frak{sl}_n\oplus\frak{sl}_m (m\ge 2, n\ge m+2); \{\omega_1, \omega_{n-1}\}\oplus\{\omega_1^* (=\omega_{n-1}), \omega_{n-1}^* (=\omega_1)\}\otimes\{\omega_1, \omega_{m-1}\}) [\mathbb Fh_{1,-1}]$;\\
11) $(\frak{sl}_n\oplus\frak{sp}_{2m}\oplus\frak{sl}_2 (n\ge 3, m\ge 1); (\{\omega_1, \omega_{n-1}\}\oplus\omega_1)\otimes\omega_1) [\mathbb Fh_{0,1}]$;\\
12) $(\frak{sl}_2; \{\omega_1\oplus\omega_1\}) [0]$;\\
13) $(\frak{sl}_n\oplus\frak{sl}_n (n\ge 2); \{\omega_1, \omega_{n-1}\}\oplus\{\omega_1^{(*)}, \omega_{n-1}^{(*)}\}\otimes\{\omega_1, \omega_{n-1}\}) [0]$;\\
14) $(\frak{sl}_{n+1}\oplus\frak{sl}_n (n\ge 2); \{\omega_1, \omega_n\}\oplus\{\omega_1^{(*)}, \omega_n^{(*)}\}\otimes\{\omega_1, \omega_{n-1}\}) [0]$;\\
15) $(\frak{sl}_2\oplus\frak{sp}_{2n} (n\ge 2); \omega_1\otimes(\mathbb F\oplus\omega_1)) [0]$;\\
16) $(\frak{sp}_{2n}\oplus\frak{sp}_{2m}\oplus\frak{sl}_2; (\omega_1\oplus \omega_1)\otimes\omega_1) [0]$;\\
17) $(\frak{sl}_2\oplus\frak{sl}_2\oplus\frak{sl}_2, (\omega_1\oplus\omega_1)\otimes\omega_1) [0]$;\\
18) $(\frak{so}_8, \{\omega_1\oplus\omega_3, \omega_1\oplus\omega_4, \omega_3\oplus\omega_4\}) [0]$.}
\section{Acknowledgements}
I thank my scientific advisor Ivan Penkov for his attention to my work and the great help with the text-editing. I thank also Dmitri Panyushev for his thoughtful comments on the paper and Vladimir Zhgoon for fruitful discussions on invariant theory. Finally, I gratefully acknowledge extensive and very helpful comments of two referees.

\end{document}